\documentclass[11pt]{amsart}

\usepackage{latexsym,amssymb,amsmath,graphics}
\usepackage[dvips]{graphicx}

\newcommand\blfootnote[1]{%
	\begingroup
	\renewcommand\thefootnote{}\footnote{#1}%
	\addtocounter{footnote}{-1}%
	\endgroup
}

\def\ff{{\mathcal F}}
\def\hh{{\mathcal H}}

\def\lll{{\mathcal L}}

\def\qq{{\mathcal Q}}

\DeclareMathOperator{\Tr}{Tr}

\def\B{{\mathbb{B}}}

\def\N{{\mathbb{N}}}

\def\R{{\mathbb{R}}}
\def\S{{\mathbb{S}}}

\newcommand{\jc}[1]{j^{(\alpha)}_{#1,c}}

\def\d{\,{\mathrm{d}}}
\newcommand{\wJ}{\widetilde P^{(\alpha,\beta)}}
\newcommand{\J}{P^{(\alpha,\beta)}}
\newcommand{\ps}{\psi^{(m,c)}_{k,\ell}}
\newcommand{\jp}{P^{(m)}_{k,\ell}}
\newcommand{\norm}[1]{{\left\|{#1}\right\|}}
\newcommand{\g}[1]{\textbf{#1}}

\newtheorem{lemma}{Lemma}[section]
\newtheorem{proposition}[lemma]{Proposition}
\newtheorem{theorem}[lemma]{Theorem}

\theoremstyle{definition}
\newtheorem{definition}[lemma]{Definition}

\theoremstyle{remark}
\newtheorem{remark}[lemma]{Remark}

\begin{document}

\blfootnote{This work was supported in part by the  
	DGRST  research grant  LR21ES10 and the PHC-Utique research project 20G1503.}

\title[Ball prolate spheroidal wave functions]{Further properties of ball prolates and approximation of related almost band-limited functions}
\author{Ahmed Souabni}
\address{Ahmed Souabni
\noindent Address: University of  Carthage,
Department of Mathematics, Faculty of Sciences of Bizerte, Bizerte, Tunisia.}
\email{souabniahmed@yahoo.fr}

\begin{abstract}
	In this paper we aim to give various explicit and local estimates of ball prolate spheroidal wave functions defined in \cite{Wang3} as eigenfunctions of both finite Fourier transform and some differential operator. In particular, we give further refined bounds of these functions and their related eigenvalues. As consequence, we show that ball PSWFs are well adapted for the approximation of almost band-limited functions and we compare this result with the one related to the ball polynomials. 
\end{abstract}
\maketitle

\section{Introduction}
Time-limited and band-limited functions are fundamental tools in signal processing. By Heisenburg's uncertainty principle, a signal can not be time and band-limited simultaneously. That is why a natural assumption is that a signal is almost band-limited. This issue has been initially carried throught Landau, Pollak and Slepian since their pioneer work in the 1960's, where prolate spheroidal wave functions have been introduced as the optimal orthogonal system to represent almost band-limited functions \cite{Hogan} \cite{landau1} \cite{landau2} \cite{Sle1}. From the investigation of the above problem, Slepian was the first to note that  PSWFs are the eigenfunctions of the finite Fourier transform operator corresponding to the eigenvalue $\lambda$, i.e
$$ \int_{-1}^1 e^{icxt} \psi(t) dt = \lambda \psi(x)  \qquad x\in I = (-1,1). $$
Slepian et al.\cite{Sle1} proved that the latter integral operator commutes with some Sturm-Liouville operator. Hence, PSWFs are also solutions of the second order differential equation 
$$ \Big((1-x^2)\psi'(x)\Big)' + (\chi - c^2x^2)\psi = 0 $$
recovered also by separation of variables for solving the Helmholtz equation in spherical coordinates. \\
This point is fundamental because that it is a perturbation of the Legendre's differential equation and in that way we link up PSWFs with orthogonal polynomials. \\
We are interested in the theory of prolate spheroidal wave functions because they have a wide range of applications and remarkable properties. Time-frequency concentration problem has been for a long while only considered over a finite interval. Then, it has been extended to other geometries as the disk, 3D ball, sphere, triangle have been considered. The reader may consult for example \cite{Beylkin} \cite{Khalid} \cite{Shkolnisky} \cite{Simons} \cite{Taylor}. \\
We are interested in the extension given by Slepian in \cite{Sle} where this problem has been extended to the d-dimensional case.
In contrast with the one dimensional case, the problem of time-frequency concentration over bounded higher dimension domain has not received enough attention.\\
In the first part of this work, we will be interested in the prolate spheroidal wave functions in the multidimensional ball. Note that the first who studied this issue was Slepian in \cite{Sle} by extending the finite Fourier transform to the d-dimension. Recently, in \cite{Wang3}, authors have given a very important contribution consisting on writing the Sturm-Liouville operator defining ball prolate spheroidal wave functions in a suitable form allowing to preserve the key features of the one-dimensional case. More precisely, they expressed the Sturm-Liouville operator of interest as a perturbation to the order $c^2 \norm{x}^2$ of the one defining the ball polynomials. Thus, we have all ingredients to develop spectral methods relative to the study of prolate spheroidal wave functions. We should mention here that whereas a more general context has been considered in \cite{Wang3},  the aim of this work is to give some refined bounds of the eigenvalues and eigenfunctions of the integral operator and to establish some other properties of the ball PSWFs.  \\
The second purpose of this work is to study the quality of approximation of almost band-limited functions by ball prolate spheroidal wave functions series expansions. In spite of their important properties, we can't handle ball PSWFs in  a straightforward way because there is no explicit formula to compute them. That is why one classical scheme is to compute explicitly their coefficients in terms of ball polynomials basis. Then, it is convenient to develop almost band-limited functions directly in the base of ball polynomials and see what happens with the quality of approximation in this framework. \\
Let us now be a little more specific and introduce some precise definitions and notations. \\
Let $\R^d$ be the d-dimensional Euclidean space,
$\g{x} $ will denote the column vector $(x_1,\cdots x_d)^T$.
We will denote the inner product over $\R^d$, for $\g{x},\g{y} \in \R^d$, by $ \displaystyle <\g{x},\g{y}>:= \sum_{i=1}^{d} x_iy_i $ and $ \norm{\g{x}}$ will denote the Euclidean associated norm $\norm{\g{x}} := \sqrt{<\g{x},\g{x}>} = \sqrt{x^2_1+ \cdots + x^2_d}$. \\
Ball prolate spheroidal wave functions are defined as solutions of the following concentration problem 
$$ \mbox{ Find } f=\arg\max_{f\in \mathcal{B}_c}\frac{\int_{\B^d}|f(\g{x})|^2 d\g{x}}{\int_{\R^d}|f(\g{x})|^2 d\g{x}}. $$
Here 
$$\mathcal{B}_c := \{ f\in L^2(\R^d) : \hat{f}(u) = 0 \quad \forall u \not \in \B^d(0,c) \},$$
$$ \mbox{ where } \B^d := \B^d(0,1), \quad \B^d(0,c) \mbox{ are more generally defined as }\B^d(0,c) := \{\g{x} \in \R^d  : \norm{\g{x}} \leq c\}. $$

\noindent The solutions of this problem are eigenfunctions of the finite Fourier transform given by 
$$\ff_c.f(\g{x}) = \int_{\B^d}e^{-ic<\g{x},\g{y}>}f(\g{y})d\g{y}.$$
On the other hand, ball PSWFs are also eigenfunctions of the following differential operator
$$ \lll_{c,\g{x}} = -\nabla.(1-\norm{\g{x}}^2)\nabla - \Delta_0 + c^2 \norm{\g{x}}^2,$$
 with $\nabla$ and $\Delta_0$ denote the gradient and the Laplace-Beltrami operators respectively.
For any positive real number $c$, we denote $\mu^{(m)}_k(c)$ and $\chi^{(m)}_k(c)$ the eigenvalues corresponding respectively to $\ff_c$ and $\lll_{c,\g{x}}$.
$$ \ff_c.\ps = \mu^{(m)}_k(c) \ps; \quad \lll_{c,\g{x}}.\ps = \chi^{(m)}_k(c) \ps ; \quad 1\leq l\leq \frac{2m+d-2}{m}\binom{m+d-3}{m-1}; \quad k,m \in \N. .$$
\noindent Note that by the form under which this last differential operator is given, by Zhang et al. in \cite{Wang3}, the ball PSWFs extend the orthogonal ball polynomials (c=0) 
$$ \jp(\g{x})= \widetilde{P}_k^{(0,m+\frac{d}{2}-1)}(2\norm{\g{x}}^2-1) Y_l^m(\hat{\g{x}}), \quad \g{x} \in \B^d,\quad 1\leq l\leq \frac{2m+d-2}{m}\binom{m+d-3}{m-1}, \quad k,m \in \N. $$
 Here $\widetilde{P}_k^{(\alpha,\beta)}$ are the normalized Jacobi polynomials that will be defined later. Note that this last form provide also a Bouwkamp spectral algorithm for the computation of ball PSWFs.
Our first result is an estimation of $\norm{\ps}_\infty$. \\
{\bf Theorem A :} Let $c>0$. For any integer $k$ such that $\displaystyle \chi^{(m)}_k>\max\{\frac{c^2+8}{(2m+d)}, \left(2/3\right)^6 \left(\frac{2\pi}{m+\frac{d}{2}-1}\right)^2+4(m+\frac{d}{2})(m+\frac{d}{2}-1)-4 +c^2  \}+m(m+d)$, we have
$$ \max_{\g{x}\in \B^d} \left| \ps(\g{x}) \right| \leq \frac{3 \sqrt{3\left(m+\frac{d}{2}-1\right)}}{2} \sqrt{\frac{N(d,m)}{\Omega_{d-1}}} \sqrt{\chi^{(m)}_k(c)}.$$
Here $ \displaystyle N(d,m) = \frac{2m+d-2}{m}\binom{m+d-3}{m-1}.$ and $\Omega_{d-1}$ denotes the surface area of $S^{d-1}$. \\
As mentioned before, and as application of this first part, we will give the quality of approximation of almost band-limited functions by ball PSWFs and by ball polynomials. We should mention here that this question has been solved in the one-dimensional case in \cite{JATpaper}. At first, let us define the concept of almost-band-limited function. \\
\begin{definition}
	Let $c>0$ and $\epsilon_c>0$. A function $f\in L^2(\R^d)$ is said to be $\epsilon_c$-band-limited function if 
	$$ \int_{\norm{\g{x}}>c} |\hat{f}(\g{x})|^2 d\g{x} \leq \epsilon_c^2 \norm{f}^2_{L^2(\R^d )}. $$
\end{definition}
The approximation of almost band-limited functions by ball prolate spheroidal functions and by ball polynomials are given by the two following theorems.\\
{\bf Theorem B :} Let $f\in L^2(\R^d)$ be an $\epsilon_c$-band-limited function. Then for any positive integer $N\geq \frac{ec}{4}$, we have 
	$$ \norm{f-S^{(M)}_N.f}_{L^2(\B^d)} \leq \left( 2\epsilon_c + C_{M,d} |\mu^{(M)}_N(c)| \left(\chi^{(M)}_N(c)\right)^{1/2} \right) \norm{f}_{L^2(\R^d)}, $$
	where $\displaystyle C_{M,d} = \frac{3}{2} \left(\frac{c}{(4\pi)^{1/4}}\right)^d \sqrt{\frac{3(M+\frac{d}{2}-1)}{\frac{d}{2}+1}}$ and $S^{(M)}_Nf$ is the orthogonal projection of a function $f\in L^2(\R^d)$ on the space spanned by the first ball prolate spheroidal wave functions .\\
 {\bf Theorem C :} 	Let $f\in L^2(\R^d)$ be an $\epsilon_c$ band-limited function. Then, for any positive integer $N\geq \frac{ec-\frac{d+1}{2}}{2} $, we have 
 	$$ \norm{f-\Pi^{(M)}_N.f}_{L^2(\B^d)} \leq \left( 2 \epsilon_c + C_N \left( \frac{ec}{2(N+1)+M+1+\frac{d+1}{2}} \right)^{2(N+1)+M+1+\frac{d+1}{2}}\right) \norm{f}_{L^2(\R^d)},$$
 	where $ \displaystyle C_{N,m,d}= \frac{1}{2^{2N+M+\frac{d}{2}+3}\sqrt{ec(4N+3M+d+4)}}\left[1+\frac{1}{4\ln \left( \frac{ec}{2N+M+2+\frac{d+1}{2}} \right)}\right]^{1/2}$ and $\Pi^{M}_Nf$ is the orthogonal projection of a function $f\in L^2(\R^d)$ on the space spanned by the first ball polynomials.\\
 	The conclusion to be drawn from the above two theorems is that the quality of approximation of almost band-limited functions is the same either by the ball prolates or by ball polynimals. The only advantage using prolates is that the truncation index is smallest.\\
The remainder of the paper is organized as follows. Section 2 is devoted to some preliminary results that will be useful afterwards. In section 3, we give some spectral properties of ball prolate spheroidal wave functions, namely the behaviour of the eigenvalues of the associated integral operator and some local estimates giving an upper bound of $\norm{\ps}_\infty$. We conclude, in section 4, by the quality of approximation of almost band-limited functions in the ball PSWFs basis comparing with the ball polynomials one.

\section{Mathematical preliminaries about some special functions }
In this section, we recall some important properties about some special functions, mainly, the ball polynomials. For this purpose, we introduce some preliminaries about spherical harmonics which appear in the definition of ball polynomials. Furthermore, we recall some properties about Bessel functions which will be frequently used throughout the forthcoming sections.

\subsection{Bessel functions}
For $\alpha>-\frac{1}{2}$, the Bessel functions $J_\alpha$ are the bounded solutions of the ordinary differential equation given by, (see for example \cite{Watson}),

\begin{equation*}
x^2y''+xy'+(x^2-\alpha^2)y=0,\quad x>0,
\end{equation*}
which is equivalent to :
\begin{equation*}
(xy')'+\Big(x-\frac{\alpha^2}{x} \Big)y=0.
\end{equation*}
Bounds and local estimates of $J_{\alpha}$ are frequently used in this paper. A first simple and useful local estimate is
given by, see \cite{Olenko}
\begin{equation*}
\label{boundJ}
\sup_{x\geq 0} \sqrt{x} |J_{\alpha}(x)| \leq c_{\alpha},
\end{equation*}
with \begin{equation*}
\label{constants}
c_{\alpha} =\left\{\begin{array}{ll}
\sqrt{2/\pi} &\mbox{ if } |\alpha|\leq 1/2\\
0.675\sqrt{\alpha^{1/3}+\frac{1.9}{\alpha^{1/3}}+
	\frac{1.1}{\alpha}}&\mbox{ if } \alpha >1/2.\end{array}\right.
\end{equation*}
A second well known estimate of the Bessel function is given in [\cite{NIST} p.227]  by 
\begin{equation} \label{estim1}
\displaystyle \Big|\frac{J_\alpha(x)}{x^\alpha} \Big| \leq \frac{1}{2^\alpha \Gamma(\alpha+1)}.
\end{equation}
When the argument is less than the order, another estimate of the Bessel function, is given by 
\begin{equation}
1 \leq \frac{J_\alpha(\alpha x)}{x^\alpha J_\alpha(\alpha)} \leq e^{\alpha(1-x)} \qquad \alpha>0 \quad\mbox{and} \quad 0<x\leq 1.
\end{equation}

In \cite{Paris}, the author has given a more precise inequality where it has been shown that 

\begin{equation} \label{paris}
\exp\Big[ \frac{\alpha^2(1-x^2)}{4\alpha+4}\Big]\leq \frac{J_\alpha(\alpha x)}{x^\alpha J_\alpha(\alpha)} \leq \exp\Big[ \frac{\alpha^2(1-x^2)}{2\alpha+4}\Big] \qquad \alpha>0 \quad\mbox{and} \quad 0<x\leq 1.
\end{equation}
The following inequality gives us a lower bound of $J_\alpha(\alpha)$ (we refer reader to \cite{Elbert})
$$ J_\alpha(\alpha) \geq \frac{\Gamma(1/3)}{2^{2/3}3^{1/6}\pi(\alpha+\alpha_0)^{1/3}} \qquad \alpha_0 \cong 0.0943498 $$
Thus, by combining the last two inequalities, one gets
\begin{equation} \label{L-B}
\displaystyle J_\alpha(\alpha x) \geq \frac{\Gamma(1/3)}{2^{2/3}3^{1/6}\pi(\alpha+\alpha_0)^{1/3}} x^\alpha \exp \Big[ \frac{\alpha^2(1-x^2)}{4\alpha+4}\Big].\qquad (\alpha>0 \quad\mbox{and} \quad 0<x\leq 1).
\end{equation}
The spherical Bessel functions are defined as (see \cite{NIST} p.262)
\begin{equation}
\label{eq:sphbess}
\jc{n}(x)=\sqrt{2(2n+\alpha+1)}\frac{J_{2n+\alpha+1}(cx)}{\sqrt{cx}}, \quad x\in (0,\infty).
\end{equation}
This latter set of functions satisfies the orthogonality relation,
$$
\int_{0}^{+\infty}\jc{n}(x)\jc{m}(x)\d x=\delta_{n,m}.
$$
Recall that the Hankel transform of a function $f\in L^2(0,\infty)$ is given by 
$$ \mathcal{H}^\alpha .f(x) := \int_{0}^\infty \sqrt{xy} J_\alpha(xy) f(y) dy; \quad \alpha>-1/2. $$
The Hankel transforms of the spherical Bessel functions are given by, (see for example \cite{Sle})
\begin{eqnarray}
\label{eq:besseljacobi}
\mathcal H^{\alpha}.\jc{n}(x)&=&\frac{\sqrt{2(2n+\alpha+1)}}{c}\left(\frac{x}{c}\right)^{\alpha+\frac{1}{2}}
P_n^{(\alpha,0)}\left(1-2\left(\frac{x}{c}\right)^2\right)
\chi_{[0,c]}(x) .
\end{eqnarray}
Here $P^{(\alpha,\beta)}_n$ are the Jacobi polynomials defined by 
$$ P_{n}^{(\alpha ,\beta)}(x) =\frac{(-1)^{n}}{2^{n}n!}(1-x)^{-\alpha
}(1+x)^{-\beta }\frac{d^{n}}{dx^{n}}\left[(1-x)^{n+\alpha }(1+x)^{n+\beta
}\right]. $$
Noting that the Hankel transform is an involution, one can write \eqref{eq:besseljacobi} in a more suitable form 
\begin{equation}\label{bessel:jacobi:bis}
	\int_{0}^1 y^{\alpha+1} J_{\alpha}(cxy) P_n^{(0,\alpha)}(2y^2-1)dy = (-1)^n \frac{J_{2n+\alpha+1}(cx)}{cx}.
\end{equation}

\subsection{Spherical harmonics}
The unit sphere $\S^{d-1}$ of $\R^d$ is denoted by 
$$ \S^{d-1}:= \{ \hat{\g{x}} \in \R^d : \norm{\hat{\g{x}}}=1 \}.$$

The inner product over $L^2(S^{d-1})$ is defined by
$$ <f,g>_{S^{d-1}} := \int_{\S^{d-1}} f(\hat{\g{x}}) g(\hat{\g{x}}) d\sigma(\hat{\g{x}}), $$
where  $d\sigma$ is the surface measure. \\
Let $\mathcal{H}^d_n$ be the space of harmonic homogeneous polynomials of degree n and $ N(d,n) := \dim\mathcal{H}^d_n$. It is well known that $ \displaystyle N(d,n) = \frac{2n+d-2}{n}\binom{n+d-3}{n-1}.$
Note that the radial and the angular variables of a function $H_n \in \mathcal{H}^d_n$ can be separated : $ H_n(\g{x}) = H_n(r\hat{\g{x}}) = r^n H_n(\hat{\g{x}}) $. Here $(r:=\norm{\g{x}}, \hat{\g{x}}:=\frac{x}{r})$ are the spherical coordinates of $\g{x}$
\begin{definition}
	A spherical harmonic of degree $n$ denoted $Y_n(\hat{\g{x}})$ is a harmonic homogeneous polynomial of degree $n$ in $d$ variables restricted to the unit $(d-1)$-sphere.
\end{definition}
It is well known that the spherical harmonics satisfy $$ \Delta_0.Y_n = -n(n+d-2) Y_n. $$ In other words, $Y_n$ are eigenfunctions of the angular part of the Laplace operator given by 
$$ \Delta_0 = \sum_{1\leq j < i \leq d} D^2_{ij} \qquad \mbox{where} \qquad D_{ij} = x_j \partial x_i - x_i \partial x_j . $$
The spherical harmonics of different degrees are orthogonal over the unit sphere, that is
$$ \int_{\S^{d-1}} Y_n(\hat{\g{x}}) Y_m(\hat{\g{x}}) d\sigma(\hat{\g{x}}) = 0 \qquad m\not = n. $$
Given a set of $N(d,n)$ linearly independent spherical harmonics of degree $n$, one can construct an orthonormal set $ \displaystyle \{ Y^{(n)}_i \}_{i=1}^{N(d,n)} $. Thus 
$$ \int_{\S^{d-1}} Y_i^{(n)}(\hat{\g{x}}) Y_j^{(m)}(\hat{\g{x}}) d\sigma(\hat{\g{x}}) = \delta_{n,m}\delta_{i,j}. $$
The spherical harmonics also satisfy the following addition formula
\begin{equation}\label{addition}
	\sum_{j=1}^{N(d,n)} Y_j^{(n)}\left(\hat{\g{x}}\right) Y_j^{(n)}(\hat{\g{y}}) = \frac{N(d,n)}{\Omega_{d-1}} \frac{C^{(\frac{d-2}{2})}_n\left(<\hat{\g{x}},\hat{\g{y}}>\right)}{C^{(\frac{d-2}{2})}_n(1)}; \qquad \Omega_{d-1} := \sigma(\S^{d-1}) = \frac{2 \pi^{\frac{d}{2}}}{\Gamma(\frac{d}{2})},
\end{equation}
 which shows that the ultra-spherical polynomial  $C^{(\frac{d-2}{2})}_n$  is the basic spherical harmonic in $d$ dimensions analogous to $\cos$ in the case $d=2$. \\
 Recall that the ultra-spherical polynomials are given by  
  $$ C_n^{(\lambda)}(x) := \frac{\Gamma(\lambda+1/2)}{\Gamma(2\lambda)} \frac{\Gamma(n+2\lambda)}{\Gamma(n+\lambda+1/2)} P_n^{(\lambda-\frac{1}{2},\lambda-\frac{1}{2})}(x); \quad C_n^{(\lambda)}(1) = \frac{\Gamma(n+2\lambda)}{\Gamma(2\lambda)\Gamma(n+1)}  \quad (n\geq 1).$$
 We will re-write the normalization of the ultra-spherical polynomials under the following form
 \begin{lemma} For any positive integer $n$ and under the above notations, we have for any $\hat{\g{x}}\in \S^{d-1}$
 	\begin{equation}\label{normalization}
 		\int_{\S^{d-1}} |C^{(\frac{d-2}{2})}_n(<\hat{\g{x}},\hat{\g{y}}>)|^2 d
 		\sigma(\hat{\g{y}}) =  \frac{\Omega_{d-1}}{N(d,n)}\left(C^{(\frac{d-2}{2})}_n(1)\right)^2.
 	\end{equation}
 \end{lemma}
\begin{proof}
	This proof is simply based on \eqref{addition} and the normalization of $\{Y_n\}_n$. In fact, 
	\begin{eqnarray}
			\int_{\S^{d-1}} |C^{(\frac{d-2}{2})}_n(<\hat{\g{x}},\hat{\g{y}}>)|^2 d
		\sigma(\hat{\g{y}}) &=& \left(\frac{\Omega_{d-1}}{N(d,n)}\right)^2 \left(C_n^{(\frac{d-2}{2})}(1) \right)^2 \int_{\S^{d-1}} \left( \sum_{j=1}^{N(d,n)} Y_j^{(n)}\left(\hat{\g{x}}\right) Y_j^{(n)}(\hat{\g{y}}) \right)^2 d\sigma(\hat{\g{y}}) \nonumber \\
		&=& \left(\frac{\Omega_{d-1}}{N(d,n)}\right)^2 \left(C_n^{(\frac{d-2}{2})}(1) \right)^2 \sum_{j=1}^{N(d,n)} Y_j^{(n)}\left(\hat{\g{x}}\right) Y_j^{(n)}(\hat{\g{x}}) \nonumber \\
		&=& \left(\frac{\Omega_{d-1}}{N(d,n)}\right)^2 \left(C_n^{(\frac{d-2}{2})}(1) \right)^2 \frac{N(d,n)}{\Omega_{d-1}}\nonumber
	\end{eqnarray}
\end{proof}
Let $Y_n$ be any spherical harmonic of degree $n$. Using the orthogonality of $Y_n$, together with \eqref{addition}, one gets 
$$ Y_n(\hat{\g{x}}) = \frac{N(d,n)}{C_n^{(\lambda)}(1)\Omega_{d-1}} \int_{\S^{d-1}} Y_n(\hat{\g{y}}) C^{(\frac{d-2}{2})}_n \left(<\hat{\g{x}},\hat{\g{y}}>\right) d\sigma(\hat{\g{y}}) .$$
By taking into account the normalization of ultra-spherical polynomials \eqref{normalization} and by Cauchy-Schwarz inequality, one gets
\begin{equation}\label{max_Y}
	\left|Y_n (\hat{\g{x}}) \right| \leq \sqrt{\frac{N(d,n)}{\Omega_{d-1}}}.
\end{equation}
To finish with this part, it is useful to note that the finite Fourier transform (over the unit sphere) of spherical harmonics is given by the following explicit expression given in {\cite{Dai}, \cite{An}}. For $\widehat{\g{x}}, \widehat{\g{y}} \in \S^{d-1}$ and $w>0$,
	\begin{equation}\label{fourier-Y}
	\int_{\S^{d-1}} e^{-iw <\widehat{\g{x}}.\widehat{\g{y}}> } Y_\ell^m(\widehat{\g{x}}) d\sigma(\widehat{\g{x}}) = \frac{(2\pi)^{d/2}(-i)^m}{w^{\frac{d-2}{2}}} J_{m+\frac{d-2}{2}}(w) Y_\ell^m(\widehat{\g{y}}) \quad  1\leq \ell \leq N(d,n), \quad m\in \N.
	\end{equation}

\subsection{Ball polynomials: Orthogonal polynomials on $\B^d$}

The ball polynomials are defined as 
$$ \jp(\g{x})= \widetilde{P}_k^{(0,m+\frac{d}{2}-1)}(2\norm{\g{x}}^2-1) Y_l^m({\g{x}}), \quad \g{x} \in \B^d\quad 1\leq\ell\leq N(d,m) \qquad k,m \in \N. $$
Here 
\begin{equation}\label{JacobiP}
	\wJ_{k}(x)= \frac{1}{\sqrt{h_k}}\J_k(x),\quad h_k=\frac{2^{\alpha+\beta+1}\Gamma(k+\alpha+1)\Gamma(k+\beta+1)}{k!(2k+\alpha+\beta+1)\Gamma(k+\alpha+\beta+1)}.
\end{equation}
It has been shown in \cite{Dai} that the ball polynomials are orthogonal with respect to the usual inner product. Recall also that the total degree of $\jp$ is $m+2k$. In addition, the ball polynomials are eigenfunctions of the following differential operator (see \cite{Dai} theorem 11.15):
\begin{equation}\label{vp_poly}
\lll_{\g{x}}.\jp(\g{x}) = \Big(-\nabla(I-\g{x}\g{x}^t)\nabla \Big)\jp(\g{x}) = (m+2k)(m+2k+d) \jp(\g{x}).
\end{equation}
Note that in \cite{Wang3}, authors have proven that this last operator can be written in different more suitable forms given by 
\begin{eqnarray}
\lll_\g{x} &=& -\nabla(1-\norm{\g{x}}^2)\nabla - \Delta_0 \nonumber \\
&=& -(1-r^2)\partial^2_r - \frac{d-1}{r} \partial_r + (d+1)r\partial_r-\frac{1}{r^2}\Delta_0,
\end{eqnarray}
Next, we will compute the finite Fourier transform of the ball polynomials $\jp$ in what follows 

\begin{lemma} For all $\g{y} = \tau \hat{\g{y}} \in \B^d$, we have
	\begin{equation}\label{finite_fourier}
		\int_{\B^d} e^{-ic<\g{x},\g{y}>} P_{j,\ell}^{(m)}(\g{x}) d\g{x} = \frac{(2\pi)^{d/2} (-i)^m (-1)^j}{2} \frac{J_{2j+m+d/2(c\tau)}}{(c\tau)^{\frac{d}{2}}}Y_{\ell}^{(m)}(\hat{\g{y}}).
	\end{equation}
\end{lemma}
\begin{proof}
	Let $\g{x} = \rho \hat{\g{x}}$ and $\g{y} = \tau \hat{\g{y}}$, then using \eqref{fourier-Y} together with \eqref{bessel:jacobi:bis} one gets
	\begin{eqnarray}
		\int_{\B^d} e^{-ic<\g{x},\g{y}>} P_{j,\ell}^{(m)}(\g{x}) d\g{x} &=& 
		\int_0^{1} \rho^{m+d-1}P_k^{(0,m+d/2-1)}(2\rho^2-1) \int_{\S^{d-1}} e^{-ic\rho \tau <\hat{\g{x}},\hat{\g{y}}>} Y_{\ell}^{(m)}(\hat{\g{x}}) d\sigma(\hat{\g{x}}) d\rho \nonumber \\
		&=& \frac{(2\pi)^{d/2} (-i)^m}{(c\tau)^{d/2-1}} \int_0^1 \rho^{m+d/2} J_{m+\frac{d}{2}-1}(c\rho\tau)  P_j^{(0,m+\frac{d}{2}-1)}(2\rho^2-1) d\rho.  Y_{\ell}^{(m)}(\hat{\g{y}}) \nonumber \\
		&=& \frac{(2\pi)^{d/2} (-i)^m (-1)^j}{2(c\tau)^{\frac{d-1}{2}}} \frac{J_{2j+m+d/2(c \tau)}}{\sqrt{c\tau}}Y_{\ell}^{(m)}(\hat{\g{y}}).
	\end{eqnarray}
\end{proof}

\section{Ball prolate spheroidal wave functions : Definitions and spectral properties}
\subsection{Definition and normalization}
We should mention here that the equivalent definitions of ball prolate spheroidal wave functions  given in this paper are an association of those given by Slepian in \cite{Sle} and those given recently by Zhang and co-authors in \cite{Wang3}. We introduce the ball PSWFs in the classical way, namely as solutions of an energy maximization problem, therefore as eigenfunctions of an integral operator and finally as eigenfunctions of a suitable differential operator. \\
In this work, we adopt the following definition of the Fourier transform over $\R^d$,
$$ \ff.f(\g{x})=\hat{f}(\g{x}) = \frac{1}{(2\pi)^{d/2}} \int_{\R^d} e^{-i<\g{x},\g{y}>} f(\g{y})d\g{y}. $$
Recall that, with this normalization, one has $ \norm{\hat{f}}_{L^2(\R^d)} = \norm{f}_{L^2(\R^d)}$. The inversion formula is then written as follows $$ \displaystyle f(\g{x}) = \frac{1}{(2\pi)^{d/2}} \int_{\R^d} e^{i<\g{x},\g{y}>} \hat{f}(\g{y})d\g{y}. $$
We are dealing with the issue of most concentrated band-limited functions on the unit ball, that is 
\begin{equation}\label{max-erg}
\mbox{ Find } f=\arg\max_{f\in \mathcal{B}_c}\frac{\int_{\B^d}|f(\g{x})|^2 d\g{x}}{\int_{\R^d}|f(\g{x})|^2 d\g{x}}.
\end{equation}

From \eqref{fourier-Y}, and by writing the integral in the spherical-polar coordinates $\g{x} = r \hat{\g{x}},\quad \hat{\g{x}} \in \S^{d-1}$, one gets 
$$ \displaystyle \int_{\B^d}e^{ic<\g{x},\g{y}-\g{z}>} d\g{x} =
\frac{(2\pi)^{d/2}}{\big(c\norm{\g{y}-\g{z}}\big)^{\frac{d}{2}-1} } \int_{0}^1 r^{\frac{d}{2}} J_{\frac{d}{2}-1}(cr\norm{\g{y}-\g{z}})dr.$$
On the other hand, 
\begin{eqnarray}
\displaystyle \int_{0}^1 r^{\frac{d}{2}} J_{\frac{d}{2}-1}(cr\norm{\g{y}-\g{z}})dr &=& \frac{1}{\big(c\norm{\g{y}-\g{z}}\big)^{\frac{1}{2}}} \int_{0}^1 (cr\norm{\g{y}-\g{z}})^{\frac{1}{2}} J_{\frac{d}{2}-1}(cr\norm{\g{y}-\g{z}})dr \nonumber \\
& = &
\frac{1}{\big(c\norm{\g{y}-\g{z}}\big)^{\frac{1}{2}}} \int_{\R} (r\norm{\g{y}-\g{z}})^{\frac{1}{2}} J_{\frac{d}{2}-1}(r\norm{\g{y}-\g{z}}) \Big(\frac{r}{c} \Big)^{\frac{d-1}{2}} \chi_{[0,c]}(r) dr \nonumber \\
&=& \frac{J_{d/2}\Big(c \norm{\g{y}-\g{z}})\Big)}{c \norm{\g{y}-\g{z}}} .
\end{eqnarray} 
Note that the last equality is obtained for n=0 in \eqref{bessel:jacobi:bis}.
Finally, by combining the last two equations, one gets
$$ \int_{\B^d}e^{ic<\g{x},\g{y}-\g{z}>} d\g{x} = (2\pi)^{d/2}\frac{J_{d/2}\Big(c \norm{\g{y}-\g{z}})\Big)}{\big(c\norm{\g{y}-\g{z}}\big)^{\frac{d}{2}} }. $$
We retain this result, that will be useful in the sequel, in the following lemma. 
\begin{lemma}
	For any positive real number $c$, and for any $\g{y},\g{z} \in \B^d$, we have 
	\begin{equation} \label{fourier_fini}
	\int_{\B^d}e^{ic<\g{x},\g{y}-\g{z}>} d\g{x} = (2\pi)^{d/2}\frac{J_{d/2}\Big(c \norm{\g{y}-\g{z}})\Big)}{\big(c\norm{\g{y}-\g{z}}\big)^{\frac{d}{2}} } .
	\end{equation}
\end{lemma}
Now, let $f \in \mathcal{B}_c$, 
\begin{eqnarray}
\frac{ \displaystyle \int_{\B^d}|f(\g{x})|^2 d\g{x}}{\displaystyle \int_{\R^d}|f(\g{x})|^2 d\g{x}} &= & \frac{1}{(2\pi)^d}
\frac{\displaystyle \int_{\B^d} \Big(\displaystyle \int_{\R^d}e^{i<\g{x},\g{y}>}\hat{f}(\g{y})d\g{y}\Big)\Big( \displaystyle \int_{\R^d}e^{-i<\g{x},\g{z}>}\overline{\hat{f}(\g{z})}d\g{z}\Big) d\g{x}}{\displaystyle\int_{\R^d}\hat{f}(\g{x}) \overline{\hat{f}(\g{x})} d\g{x}} \nonumber \\
&=&\Big( \frac{1}{2\pi}\Big)^{d} \frac{\displaystyle \int_{\B^d(0,c)} \Bigg(\displaystyle \int_{\B^d(0,c)} \Big(\displaystyle \int_{\B^d}e^{i<\g{x},\g{y}-\g{z}>} d\g{x} \Big) \hat{f}(\g{y}) d\g{y} \Bigg) \overline{\hat{f}(\g{z})} d\g{z} }{\displaystyle \int_{\R^d}\hat{f}(\g{x}) \overline{\hat{f}(\g{x})} d\g{x}} \nonumber \\
&=& \Big( \frac{1}{2\pi} \Big)^{d/2} \frac{\displaystyle \int_{\B^d(0,c)} \Bigg(\displaystyle \int_{\B^d(0,c)} \frac{J_{d/2}(\norm{\g{y}-\g{z}})}{\Big(\norm{\g{y}-\g{z}}\Big)^{\frac{d}{2}}} \hat{f}(\g{y}) d\g{y} \Bigg) \overline{\hat{f}(\g{z})} d\g{z} }{\displaystyle \int_{\B^d(0,c)}\hat{f}(\g{x}) \overline{\hat{f}(\g{x})} d\g{x}} .
\end{eqnarray}

Hence, by a straightforward change of variable and function, the solutions of $$ \arg\max_{f\in \mathcal{B}_c}(2\pi)^{\frac{d}{2}} \frac{ \displaystyle \int_{\B^d}|f(\g{x})|^2 d\g{x}}{\displaystyle \int_{\R^d}|f(\g{x})|^2 d\g{x}} $$ are the eigenfunctions of the integral operator $\qq_c$ given by 
\begin{equation}
\qq_c.f(\g{x}) = \displaystyle \Big( \frac{c}{2\pi} \Big)^{d} \int_{\B^d} (2\pi)^{d/2}\frac{J_{d/2}(c\norm{\g{y}-\g{z}})}{\Big(c\norm{\g{y}-\g{z}}\Big)^{\frac{d}{2}}} f(\g{y}) d\g{y} .
\end{equation}

One can easily check that $ \qq_c = \left( \frac{c}{2\pi} \right)^{d}  \ff_c^* \circ \ff_c $ where $\ff_c $ is the finite Fourier integral operator defined on $L^2(\B^d)$ by 

\begin{equation}
\ff_c.f(\g{x}) = \int_{\B^d}e^{-ic<\g{x},\g{y}>}f(\g{y})d\g{y}.
\end{equation}
It has been shown in \cite{Wang3} that
\begin{itemize}
	\item The eigenfunctions of $\qq_c$ are the same as those of the following positive self-adjoint differential operator
	\begin{eqnarray}\label{diff_op}
		\lll_{c,\g{x}} &=& -\nabla.(1-\norm{\g{x}}^2)\nabla - \Delta_0 + c^2 \norm{\g{x}}^2 \nonumber \\
		&=& -(1-r^2)\partial^2_r - \frac{d-1}{r} \partial_r + (d+1)r\partial_r-\frac{1}{r^2}\Delta_0 + c^2 r^2.
	\end{eqnarray}
	\item  Eigenfunctions of both $\lll_{c,\g{x}}$ and $\qq_c$ (or equivalently $\ff_c$) can be written under the following form 
	\begin{equation} \label{sep_form}
		\ps(\g{x}) = r^m \phi_k^{(m,c)} (2r^2-1) Y_\ell^m(\hat{x}), 
	\end{equation}
	Here $(r:=\norm{\g{x}}, \hat{\g{x}}:=\frac{x}{r})$ are the spherical coordinates of $\g{x}$, $k,m$ are two positive integers, $1\leq \ell \leq N(d,m)$ and $\phi_k^{(m,c)}$ satisfies 
	\begin{equation}\label{diff-rad}
		\frac{-1}{\omega_{0,\beta_{m,d}}(\eta)}\partial_\eta \Big(\omega_{1,\beta_{m,d}+1}(\eta) \partial_\eta \phi_k^{(m,c)}  \Big) + \frac{c^2(\eta+1)}{8} \phi_k^{(m,c)} = \alpha_{k,m}(c)\phi_k^{(m,c)}
	\end{equation}
with $\eta = 2r^2-1$, $\gamma_{m,d}=m(m+d)$ and $\omega_{\alpha,\beta}(x)=(1-x)^\alpha(1+x)^\beta.$
	\item The eigenvalues of $\lll_{c,\g{x}}$, $\ff_c$ and $\qq_c$ denoted respectively by $\chi^{(m)}_k(c)$, $\mu^{(m)}_k(c)$, $\lambda^{(m)}_k(c)$ are independent of $\ell$. 
\end{itemize}
We define the ball spheroidal wave functions as solutions of the energy maximization problem \eqref{max-erg} which are then at the same time  eigenfunctions of all operators $\qq_c$, $\ff_c$ and $\lll_{c,\g{x}}$. Consequently, one can write, for $c>0$, $k,m\in \N$ and $1\leq \ell \leq N(d,m)$
$$ \lll_{c,\g{x}}.\ps = \chi^{(m)}_k(c) \ps ; \quad \ff_{c}.\ps = \mu^{(m)}_k(c) \ps ; \quad \qq_{c}.\ps = \lambda^{(m)}_k(c) \ps, $$
where $$ \lambda^{(m)}_k = \left(\frac{c}{2\pi}\right)^d |\mu^{(m)}_k|$$
\begin{remark} \label{connection_d2}
	We use in this remark the separated form of the ball prolate spheroidal wave functions to show that the radial part of these functions are also eigenfunctions of the finite Hankel transform. We should mention that this remark has been given differently in \cite{Sle}. \\
Thanks to \eqref{sep_form} and \eqref{fourier-Y}, one can write
$$ \lambda_k^{(m)} r^m \phi_k^{(m,c)} (2r^2-1) = (2\pi)^{d/2} (-i)^m \int_0^1 \frac{J_{m+\frac{d}{2}-1}(cr\tau)}{(cr\tau)^{\frac{d}{2}-1}} \tau^{m+d-1}\phi_k^{(m,c)} (2r^2-1) d\tau .  $$
It is convenient to make the substitution $ \varphi_k^{(m,c)}(r) := r^{m+\frac{d-1}{2}}\phi_k^{(m,c)} (2r^2-1) $ in order to obtain 
$$ \lambda_k^{(m)} \varphi_k^{(m,c)}(r) \sqrt{c}\big( \frac{c}{2\pi} \big)^{d/2} (i)^m
= \int_{0}^1 J_{m+\frac{d}{2}-1}(cr\tau)
\sqrt{cr\tau}\varphi_k^{(m,c)}(\tau) d\tau . $$
Then 
\begin{equation}{\label{relation2}}
 \alpha_k^{(m)} \varphi_k^{(m)}(r) = \int_{0}^1 J_{\alpha}(cr\tau) \sqrt{cr\tau}\varphi_k^{(m)}(\tau) d\tau=\mathcal{H}^{(\alpha)}_c.\varphi_k^{(m)}(r) ; \quad \alpha=m+\frac{d}{2}-1 ; \quad \alpha_k^{(m)}= \sqrt{c}\big( \frac{c}{2\pi}\big)^{d/2} \lambda_k^{(m)}(c) .
\end{equation}
 Here $\mathcal{H}^{(\alpha)}_c $ is the finite Hankel transform given by 
 $$ \mathcal{H}^{(\alpha)}_c.f(x) = \int_{0}^1 \sqrt{cxy} J_\alpha(cxy) f(y)dy .$$
 It may be useful to note that the problem of the behavior of eigenvalues of the finite Hankel transform has been extensively studied, see for example \cite{Karoui-Boulsane}, \cite{Karoui-Moumni-1} and \cite{Boulsane}. 
\end{remark}
	 We note finally that the ball PSWFs are normalized through the following rule :
	 $$ \int_{\B^d} \big(\ps(\g{x})\big)^2 d\g{x} = 1, $$
	 or equivalently, in terms of the radial part:
	 $$ \int_{0}^1 r^{2m+d-1} \Big(\phi_k^{(m,c)} (2r^2-1)\Big)^2 dr = 1 \quad \mbox{or} \quad \int_{-1}^1 (1+t)^{m+\frac{d}{2}-1} |\phi_k^{(m,c)}|^2(t) dt = 2^{m+\frac{d}{2}-1}.  $$
	 Finally, we write the ball prolate in terms of limiting operators : 
	 Let $U$ be a set of finite measure in $\R^d$ and $\mathcal{D}(U)$ be the subspace of $L^2(\R^d)$ formed by the functions supported in $U$,
	 $$ \mathcal{D}(U) = \{ f \in L^2(\R^d) : f(x)=0 \quad \forall x \not \in U \},  $$
	 and recall that $\mathcal{B}_c$ is the Paley-Wiener space formed by functions whose Fourier transform are supported in $ \B^d(0,c)$,
	 $$ \mathcal{B}_c = \{ f\in L^2(\R^d) : \mathcal{F}.f(u) = 0 \quad \forall u \not \in \B^d(0,c) \}. $$
	 Let $ D_U.f(x) = \chi_U(\g{x}) f(\g{x}) $ be the orthogonal projection of $L^2(\R^d)$ onto $\mathcal{D}(U)$ and 
	 $ B_c.f(x) = \mathcal{F}^{-1} \chi_{\B^d(0,c)}\mathcal{F}.f(x) $ 
	 be the orthogonal projection of $L^2(\R^d)$ onto $\mathcal{B}_c$.\\
	 Using \eqref{fourier_fini}, 
	 $$ B_c.f(x) = \Big(\frac{c}{2\pi}\Big)^{d/2} \int_{\R^d} (2\pi)^{d/2} \frac{J_{d/2}(c\norm{\g{x}-\g{y}})}{(c\norm{\g{x}-\g{y}})^{d/2}}  f(\g{y})d\g{y} $$
	 Then, one can write the integral operator $\mathcal{Q}_c$ in terms of the limiting operators as 
	 $$ \displaystyle \mathcal{Q}_c = D_{\B^d}B_c D_{\B^d}$$
	 \subsection{Behaviour of $\lambda^{(m)}_n(c)$ }

Recently in \cite{Boulsane}, the author has given a well precise behaviour of the eigenvalues $\alpha_k^{(m)}$ which are related, by remark \ref{connection_d2}, to the eigenfunctions corresponding to ball prolate functions. Mainly, one can derive an asymptotic super-exponential decay of $\lambda_n^{m}(c)$ directly by  writing theorem 3.2 of \cite{Boulsane} under our notation :

\begin{proposition}
	For given real numbers $m,c>0$, there exists a constant $A$ depending only on $c$, such that for every $n>\frac{ec}{4}$, we have 
	\begin{equation}
	\lambda_n^{(m)}(c) \leq A  \Bigg[\frac{ec}{4n+2m+d}\Bigg]^{4n+2m+d}.
	\end{equation}
	
\end{proposition}

In the following proposition, we provide the reader with a lower bound of the eigenvalues $\lambda_n^{m}(c)$.

\begin{proposition}
	For any positive real number $c>0$ and for any $n\geq \frac{ec}{4}$, we have
	$$ \lambda_n^{m}(c) \geq C \Big( \frac{ec}{4n+2m+d}\Big)^{4n+2m+d}, $$
	where $C = \displaystyle \Big(\frac{\Gamma(1/3)}{2^{2/3}3^{1/6} \pi (2k+\alpha+\alpha_0+1)^{1/3}}\Big)^2 \frac{\left(2\pi\right)^{d/2}}{c^{\frac{d+1}{2}}} $, $\alpha$ and $\alpha_0$ are defined in \eqref{relation2} and \eqref{L-B} respectively. 
\end{proposition}
\begin{proof}
 We will prove this lower bound for the eigenfunctions of the finite Hankel transform and use again the remark giving the relation between the eigenfunctions of the ball PSWFs and those of the circular ones \eqref{relation2}. For this purpose, we recall that the eigenfunctions of $\qq^{(\alpha)}_c:=c\hh^{(\alpha)}_c \circ \hh^{(\alpha)}_c $ are characterized, using Courant-Fischer max-min theorem, by 
 $$ \alpha^{(m)}_n (c) = \max_{V\in G_n} \min_{y\in V;\norm{y}=1} <\qq^{(\alpha)}_c y;y> $$
 Since $\left(T^{(\alpha)}_n :=\sqrt{2(2n+\alpha+1)}x^{\alpha+1/2}P_n^{(0,\alpha)}(2x^2-1)\right)_n$ is an orthonormal basis of $L^2(0,1)$, then
$$ \alpha^{(m)}_n(c) \geq <\qq^{(\alpha)}_c T^{(\alpha)}_n,T^{(\alpha)}_n>_{L^2([0,1])} . $$
 On the other hand, using \eqref{eq:besseljacobi} together with \eqref{eq:sphbess}, one gets
 \begin{equation*}
 \displaystyle
 \norm{\hh^{(\alpha)}_c.T^{(\alpha)}_k}^2_2 = \norm{j^{(\alpha)}_k}^2_2 = 2(2k+\alpha+1) \int_{0}^1 \Bigg| \frac{J_{2k+\alpha+1}(cx)}{\sqrt{cx}}\Bigg|^2 dx.
 \end{equation*}
 Thanks to \eqref{L-B}, one can write,
$$
J_{2k+\alpha+1}(cx)\geq \frac{\Gamma(1/3)}{2^{2/3}3^{1/6} \pi (2k+\alpha+\alpha_0+1)^{1/3}}
e^{-(2k+\alpha+1) \Big(\ln(2k+\alpha+1)-1/4 \Big) } (cx)^{2k+\alpha+1} e^{\frac{-(cx)^2}{4(2k+\alpha+1)}} $$
Therefore, 
\begin{eqnarray}
\norm{\hh^{(\alpha)}_c.T^{(\alpha)}_k}_2  &\geq& C^2 
e^{-2(2k+\alpha+1) \Big(\ln(2k+\alpha+1)-1/4 \Big) } \int_{0}^c x^{2(2k+\alpha+1/2)}e^{-\frac{x^2}{2(2k+\alpha+2)}} dx  \nonumber \\
&\geq& C^2 
e^{-2(2k+\alpha+1) \Big(\ln(2k+\alpha+1) \Big) } \int_{0}^c x^{2(2k+\alpha+1/2)}e^{-\frac{x^2}{2(2k+\alpha+2)}} dx \nonumber \\
&\geq& C^2 
\Big( \frac{ec}{2(2k+\alpha+1)}\Big)^{2k+\alpha+1}. \nonumber
\end{eqnarray} 
To conclude for the proof, it suffices to use remark \ref{connection_d2}.
\end{proof}

We will give in the next proposition a brief description, to the first order, of the counting number of the eigenvalues $\lambda^{(m)}_n(c) $.We use here the well known Landau's technique \cite{Landau} based on computing the trace and the Hilbert-Schmidt norm of $\mathcal{Q}_c$.
\begin{proposition}
	Let $0<\delta<1$ and let $M_c(\delta)$ denote the number of eigenvalues $\lambda_k(c)$ not smaller than $\delta$. Then 
	$$ M_c(\delta) = \frac{c^d}{2^d} \frac{1}{\Gamma^2(\frac{d}{2}+1)} + o(c^d). $$
\end{proposition}
\begin{proof}
We start by computing the trace of $\mathcal{Q}_c$ by using Mercer's theorem together with the fact that $\displaystyle K(0)=\frac{1}{2^{d/2}\Gamma(\frac{d}{2}+1)}$ and $\displaystyle \mu(\B^d) = \frac{\pi^{d/2}}{\Gamma(\frac{d}{2}+1)}$ (here $\mu$ denotes the Lebesgue measure), 
\begin{equation}\label{trace}
	\Tr(\mathcal{Q}_c) = \displaystyle \sum_n \lambda^{(m)}_n(c) = \Big(\frac{c}{\sqrt{2\pi}}\Big)^d \mu(\B^d) K(0) = \frac{c^d}{2^d\Gamma^2(\frac{d}{2}+1)}.
\end{equation}
Here $K(x):= \frac{J_{d/2}(\norm{x})}{(\norm{x})^{d/2}}$.
On the other hand, to derive an estimate of  $ \norm{\mathcal{Q}_c}_{HS}$, we proceed as follows
$$ \norm{\mathcal{Q}_c}_{HS} = \displaystyle \sum_n \Big(\lambda^{(m)}_n(c)\Big)^2 = \Big(\frac{c}{\sqrt{2\pi}}\Big)^{2d}\int_{\B^d}\int_{\B^d} |K(c\norm{\g{x}-\g{y}})|^2 d\g{x} d\g{y}. $$
By applying the change of variables $y=\sigma$ and  $x=\sigma+\frac{\tau}{c}$, one gets 
$$ \norm{\mathcal{Q}_c}_{HS} = \Big( \frac{c}{2\pi} \Big)^d \int_{\B^d} \Bigg( \int_{c(\B^d-\sigma)} |K({\tau})|^2 d\tau \Bigg) d\sigma. $$
Note that
\begin{eqnarray}
	\int_{c(\B^d-\sigma)} |K(\tau)|^2 d\tau &\leq& \int_{\R^d} |K(\tau)|^2 d\tau = \int_{\R^d} \frac{J_{d/2}^2(\norm{\tau})}{\norm{\tau}^d} d\tau \nonumber \\
	&=& \frac{2\pi^{d/2}}{\Gamma(d/2)} \int_{0}^{\infty} \frac{J^2_{d/2}(t)}{t} dt = \frac{\pi^{d/2}}{\Gamma(\frac{d}{2}+1)}.
\end{eqnarray}
The last equality is due to [\cite{NIST} eq 22.58 page 244].
Moreover, $\B^d-\sigma$ contains some $B(0,\alpha).$ Then, by Lebesgue's dominated convergence theorem, one has 
$ \displaystyle \lim_{c\to \infty} \frac{1}{c^d} \norm{\mathcal{Q}_c}_{HS} = \frac{1}{2^d \Gamma(\frac{d}{2}+1)} $. That is 
\begin{equation}\label{norm}
	\norm{\mathcal{Q}_c}_{HS} = \frac{c^d}{2^d \Gamma(\frac{d}{2}+1)} + o(c^d).
\end{equation}
Next, we notice that 
\begin{equation}\label{inf}
	\Tr(\mathcal{Q}_c) \geq \sum_{k=0}^{M_c(\delta)}\lambda_k \geq \delta M_c(\delta),
\end{equation}
and using Marzo's formula (see \cite{Marzo}), one gets
\begin{equation}\label{sup}
	M_c(\delta) \geq \Tr(\mathcal{Q}_c)-\frac{1}{1-\delta} \Big(\Tr(\mathcal{Q}_c)-\norm{\mathcal{Q}_c}_{HS}\Big).
\end{equation}
Let $\displaystyle M_+(\delta) := \limsup_{c\to \infty} \frac{M_c(\delta)}{c^d}$ and $\displaystyle M_-(\delta) := \liminf_{c\to \infty} \frac{M_c(\delta)}{c^d}$.\\
By combining together \eqref{trace}, \eqref{norm},\eqref{inf} and \eqref{sup}, one gets
\begin{equation}\label{enc}
	\frac{1}{2^d} \frac{1}{\Gamma^2(\frac{d}{2}+1)}  \leq M_- (\delta)\leq M_+(\delta) \leq \frac{1}{2^d} \frac{1}{\delta \Gamma^2(\frac{d}{2}+1)} \quad \forall 0<\delta<1 .
\end{equation}
 The next step is to prove that both $M_+$ and $M_-$ are independent of $\delta$ for all $0<\delta<1$. For this claim, we start by noticing that the number of eigenvalues not close to 0 or 1 is $o(c^d)$ by considering 
$$ J_c := \sum_{k=0}^\infty \lambda_k(c)(1-\lambda_k(c)) = \Tr(\mathcal{Q}_c) - \norm{\mathcal{Q}_c}_{HS} = o(c^d).  $$
Now let $\delta$ and $\gamma$ be two fixed reals such that $ 0<\delta<\gamma<1$. Taking into account that each eigenvalue $ \delta<\lambda_k(c) < \gamma$ has contribution to $J_c$ by at least $\delta(1-\gamma)$, one has 
$$\delta(1-\gamma) \left[ M_c(\delta) - Mc(\gamma) \right] \leq J_c = o(c^d)$$ and consequently $M_+$ and $M_-$ are both independent of $\delta$.
To conclude for the proof it suffices to choose $\delta$ near $1$ on the right side at \eqref{enc}.
\end{proof}
\subsection{Further basic properties}
We study here some other basic properties of the ball prolate spheroidal wave functions. \\
First, we give the bounds of the eigenvalues corresponding to the Sturm-Liouville operator. This lemma has been given in \cite{Wang3}. Here we prove it by other means.
\begin{lemma}
	For any positive real number $c$, we have 
	$$ (m+2k)(m+2k+d) \leq \chi^{(m)}_k(c) \leq (m+2k)(m+2k+d)+c^2; \qquad k\in \N, \quad m\in \N. $$
\end{lemma}

\begin{proof}
	To get the required upper bound, we write the differential operator $\lll_{c,\g{x}}$ as 
	\begin{eqnarray*}
		\lll_{c,\g{x}}.u(x)&=&-\nabla(1-\norm{\g{x}}^2)\nabla.u(x)-\Delta_0.u(x)+c^2 \norm{\g{x}}^2u(x) \\
		&=& \lll_{0,\g{x}}.u(x)+c^2\norm{\g{x}}^2u(x)
	\end{eqnarray*}
	Then, by using the Courant-Fischer min-max theorem applied to the eigenvalues of the self-adjoint operator $\lll_{c,\g{x}}$, one gets 
	\begin{eqnarray*}
		\chi_k^{(m)}(c) &=& \min_{\dim H =k} \max_{u\in H ; \norm{u}=1} <\lll_{c,\g{x}}.u,u > \\
		&\leq& \min_{\dim H =k} \max_{u\in H ; \norm{u}=1} <\lll_{0,\g{x}}.u,u >+c^2\norm{u}^2 \\
		&\leq& \chi_k^{(m)}(0)+c^2.
	\end{eqnarray*}
On the other hand the lower bound follows from the fact that $\lll_{c,\g{x}}-\lll_{0,\g{x}}=c^2\norm{\g{x}}^2$ is a positive operator. Finally, to conclude for the proof of this lemma, it suffices to use the expression of $\chi_k(0)$given by \eqref{vp_poly}.
\end{proof}

The next proposition allows us to keep the fundamental property of double orthogonality already seen in the classical case. This is provided by computing the Fourier transform of the ball prolates spheroidal wave functions.

\begin{lemma}\label{Fourier_inv}
	The Fourier transform of the ball PSWFs are given by 
	$$ \displaystyle \ff.\ps(x) = \frac{(2\pi)^d}{c^d \sqrt{\lambda_k^{(m)}(c)}} \ps\big(\frac{-x}{c}\big) \chi_{(\B^d)} \big(\frac{x}{c}\big).$$
\end{lemma}

\begin{proof}
	By the inverse Fourier transform, one has for $ f\in L^2(\R^d)$, 
	\begin{eqnarray}\label{inv_bis}
		f(\g{x}) &=& \frac{1}{\sqrt{(2\pi)^d}} \int_{\R^d} e^{i<x,z>} \ff.f(z) d\g{z} = \frac{1}{(2\pi)^d} \int_{\R^d} e^{i<x,z>} \int_{\R^d} e^{-i<z,y>} f(\g{y})d\g{y} d\g{z} \nonumber \\
		&=&  \frac{1}{(2\pi)^d}\int_{\R^d} \int_{\R^d} e^{-i<\g{z},(\g{y}-\g{x})>} f(\g{y}) d\g{y} d\g{z}.
	\end{eqnarray}
	On the other hand, from
	$$ \ps(\g{x}) = \frac{1}{\mu^{(m)}_k(c)} \int_{\B^d}e^{-ic<\g{x},\g{y}>} \ps(\g{y})d\g{y}, $$
	one gets 
	\begin{eqnarray}
	\ff.\ps(\g{x})&=&\frac{1}{\mu^{(m)}_k(c)}\int_{\R^d} \int_{\B^d} e^{-ic<\g{z},\g{y}>}\ps(\g{y})dy e^{-i<\g{z},\g{x}>} d\g{x} \nonumber \\
	&=& \frac{1}{\mu^{(m)}_k(c)}\int_{\R^d} \Bigg( \int_{\R^d} e^{-ic<\g{z},\g{y}-\g{x}>}\ps(\g{-y}/c)\chi_{(\B^d)}(\g{y}/c)d\g{y}\Bigg)  d\g{x} .
	\end{eqnarray}
To conclude for the proof it suffices to use the last equation together with \eqref{inv_bis}.
\end{proof}

\subsection{Some explicit estimates and bounds of eigenfunctions}
The purpose of this paragraph is to give an explicit upper bound of the ball prolate spheroidal wave functions. To do so, we start by showing that under some conditions, these functions reach their maximum on the unit sphere $\S^{d-1}$.
Let $c>0$, we recall that 
$$ \ps(\g{x}) = r^m \phi_k^{(m,c)}(\eta) Y_\ell^{(m)}(\hat{\g{x}}), \quad \g{x} = r\hat{\g{x}}, \quad r=\norm{\g{x}}, \quad \eta:=2r^2-1. $$ 

\begin{lemma}
	Let $\alpha^{(m)}_k(c) := \frac{1}{4}\left( \chi^{(m)}_k(c)-m(m+d) \right) $ and $c$ be a positive real number. If $\alpha^{(m)}_k(c) > \frac{c^2}{4}$, then we have 
	\begin{equation}
		\sup_{\eta \in [a_{m,d},1]} |\phi^{(m)}_k(\eta) | = |\phi^{(m)}_k(1)| \qquad a_{m,d} = \frac{2m+d-2}{2m+d}.
	\end{equation}
\end{lemma}

\begin{proof}
	For the sake of clarity, we write, throughout the proofs of this section, $\phi_k$ and $\alpha_k$ instead of $\phi^{(m,c)}_k$ and $\alpha^{(m)}_k(c)$ respectively. Hence \eqref{diff-rad} can be written under the form 
	\begin{equation}\label{gen_form}
	\left( p(\eta) \phi_k'(\eta) \right)' + q_k(\eta) \phi_k(\eta) = 0;  \quad p(\eta) = (1-\eta)(1+\eta)^{m+\frac{d}{2}}, \; q_k(\eta) = \alpha_k\left( 1- \frac{c^2(1+\eta)}{8\alpha_k}\right)\left(1+\eta\right)^{m+\frac{d}{2}-1}.
	\end{equation}
	Straightforward computations show that the auxiliary function $Z_k(\eta):= \phi_k^2(\eta) + \frac{p(\eta)}{q_k(\eta)}\phi_k'^2(\eta) $ admits a first order derivative only on $\phi_k'^2$, given by
	$$ Z'_k(\eta) = -\frac{1}{q^2_k(\eta)}\left(p(\eta)q_k(\eta)\right)' \phi_k'^2(\eta) .$$
	Then, in our case
	$$ Z'_k(\eta) = -\frac{1}{q^2_k(\eta)}(1+\eta)^{2m+d-2}\Bigg[ \Big(\alpha_k-\frac{c^2(1+\eta)}{8}\Big) \Big(2m+d-2-t(2m+d)\Big) - \frac{c^2}{8} (1-\eta^2) \Bigg] \phi_k'^2(\eta). $$
	To conclude for the proof, it suffices to note that under the condition $\alpha^{(m)}_k(c) > \frac{c^2}{4}$, $Z_k$ is increasing over $[a_{m,d},1]$ and $$\Big|\phi_k^2(\eta)\Big| \leq Z_k(1) \leq \Big|\phi_k^2(1)\Big|. $$
\end{proof}
\begin{lemma}
	Under conditions of the previous lemma, we have 
	\begin{equation}\label{local_estimate}
		\sup_{\eta \in [a_{m,d},1]} \sqrt{(1-\eta)(1+\eta)^{m+\frac{d}{2}}} \big|\phi_k(\eta)\big| \leq \sqrt{2^{m+\frac{d}{2}+1}(m+\frac{d}{2}-1)}.
	\end{equation}
\end{lemma}
\begin{proof}
	The proof of this lemma is based on a fairly well known technique for the Sturm-Liouville theory consisting on the use of the auxiliary function defined by $K_k(\eta)=p(\eta)Z_k(\eta)$. \\
	Straightforward computations show that 
	$$ K'_k(\eta) = p'(\eta)\phi_k^2(\eta) - \frac{p^2(\eta)q'(\eta)}{q^2(\eta)} (\phi'_k)^2(\eta). $$
	Then, since  $\alpha_k > \frac{c^2}{4}$ and $\eta\geq \frac{2m+d-2}{2m+d}$, one has
	$$ K'_k(\eta) \geq (1+\eta)^{m+\frac{d}{2}-1} \Big[\big(m+\frac{d}{2}-1\big)-\big(m+\frac{d}{2}+1\big)\eta\Big]\phi_k^2(\eta).$$
	Hence, taking into account that $\eta<1$ together with the normalization of $\phi_k$,
	$$ K_k(\eta) = K_k(\eta)-K_k(1) \leq 2\int_{-1}^{1}(1+t)^{m+\frac{d}{2}-1}\phi_k^2(t)dt = 2^{m+\frac{d}{2}+1} (m+\frac{d}{2}-1). $$
	Finally, one has
	$$ (1-\eta)(1+\eta)^{m+\frac{d}{2}} \left| \phi_k(\eta) \right|^2 \leq K_k(\eta) \leq 2^{m+\frac{d}{2}+1} (m+\frac{d}{2}-1). $$
\end{proof}
\begin{proposition}
	Let $c$ be a real positive number. If $\alpha^{(m)}_k(c)> \frac{c^2}{4}$, then 
	\begin{equation}
		\sup_{\eta \in [a_{m,d},1]} \big|\phi_k(\eta) \big| \leq \frac{3\sqrt{3}}{2}\sqrt{2^{m+\frac{d}{2}+1}(m+\frac{d}{2}-1)} \big(\alpha^{(c)}_{k,m}\big)^{1/2}
	\end{equation}
\end{proposition}
\begin{proof}
	Without loss of generality, one may assume that $\ps(1)>0$. By \eqref{diff-rad}, one has
	$$\Big((1-\eta^2)(1+\eta)^{m+\frac{d}{2}-1} \phi_k'(\eta)\Big)'=-\alpha^{(c)}_{km} \Big(1-\frac{c^2(\eta+1)}{8\alpha^{(c)}_{km}}\Big) (1+\eta)^{m+\frac{d}{2}-1}\phi_k(\eta).$$
	Then, by integrating over $[x,1]$, one gets 
	$$ \phi_k'(\eta) = \frac{\alpha_{k}}{(1-\eta^2)(1+\eta)^{m+\frac{d}{2}-1}} \int_{\eta}^{1} (1+t)^{m+\frac{d}{2}-1}\Big(1-\frac{c^2(t+1)}{8\alpha_{k}}\Big) \phi_k(t)dt .$$
	Straightforward computations show that 
	$$ q_k'(\eta) = (1+\eta)^{m+\frac{d}{2}}\Big[\alpha_{k}(m+\frac{d}{2}-1)-\frac{c^2}{8}(m+\frac{d}{2}) - \frac{c^2}{8}(m+\frac{d}{2})\eta \Big] $$
	Thus, for $\alpha_{k}>\frac{c^2}{4}$, $q$ is decreasing over $[a_{m,d},1]$ and
$$ \phi_k'(\eta) \leq \frac{\alpha_{k}}{1-\eta^2} \Big(1-\frac{c^2(\eta+1)}{8\alpha_{k}}\Big)(1-\eta)\phi(1) =  \alpha_{k} \Big(1-\frac{c^2(\eta+1)}{8\alpha_{k}}\Big)\phi_k(1).  $$
Then, 
$$ \phi_k(1)-\phi_k(\eta) \leq \alpha_{k}\phi_k(1) \Big(1-\frac{c^2(\eta+1)}{8\alpha_{k}}\Big)(1-\eta) .  $$
Let $x_k$ be in a neighbourhood of 1 such that $\Big(1-\frac{c^2(x_k+1)}{8\alpha_{k}}\Big){(1-x_k)} = \frac{A}{\alpha_{k}}$. Consequently,
$$ \phi_k(1) \leq \frac{1}{1-A} \frac{\sqrt{2^{m+\frac{d}{2}+1}\left(m+\frac{d}{2}-1\right)}}{\sqrt{(1-x_n)(1+x_n)^{m+d/2}}} \leq \frac{\sqrt{2^{m+\frac{d}{2}+1}\left(m+\frac{d}{2}-1\right)}}{(1-A)A^{1/2}}(\alpha_{k})^{1/2}. $$
One concludes for the proof by noticing that $ \displaystyle \min_{A} \frac{1}{A^{1/2}(1-A)} = \frac{3\sqrt{3}}{2}$.
 
\end{proof}

\begin{theorem}\label{max}
	Let $c>0$ and $\alpha^{(m)}_k(c)> \max \{\frac{c^2+8}{4(2m+d)}, \left(2/3\right)^6 \left(\frac{\pi}{m+\frac{d}{2}-1}\right)^2+C_{m,d}(c) \}$. Then 
	$$ \max_{\g{x}\in \B^d} \left| \ps(\g{x}) \right| \leq \frac{3\sqrt{3}}{2}\sqrt{2^{m+\frac{d}{2}+1}(m+\frac{d}{2}-1)} \sqrt{\frac{N(d,m)}{\Omega_{d-1}}}. \sqrt{\chi^{(m)}_k(c)}, $$
	with $ C_{m,d}(c) = \frac{c^2}{4}+(m+\frac{d}{2})(m+\frac{d}{2}-1)-1.$
\end{theorem}
\begin{proof}
We start by recalling Butlowski's theorem concerning the behavior of the local extrema of the solution of a second order differential equation (see \cite{An} p238) : \\
	 If $\phi$ is a solution of the differential equation 
	$$ (p(t)y'(t))'+q(t)y(t)=0 \qquad t\in(a,b) \qquad p,q>0 \mbox{ and }p,q\in C^1(a,b), $$
	then the sequence of local maximum of $|\phi|$ is decreasing (respectively increasing) if $pq$ is increasing(respectively decreasing).

In our case, we recall that 
$$ \Big(p(\eta)q_k(\eta)\Big)' = (1+\eta)^{2m+d-2}\Bigg[ \Big(\alpha_k-\frac{c^2(1+\eta)}{8}\Big) \Big(2m+d-2-\eta(2m+d)\Big) - \frac{c^2}{8} (1-\eta^2) \Bigg]=(1+\eta)^{2m+d-2} F_k(\eta) .$$
Then, it is easy to see that $F_k$ has a unique zero in $[0,1]$. Then there exists a unique real number $\eta_n$ so that $pq$ is increasing on $[0,\eta_n]$ and decreasing on $[\eta_n,1]$. Hence, the local maxima of $\phi_k$ are decreasing on $[0,\eta_n]$ and increasing on $[\eta_n,1]$. Let $\eta'_{1,n}$ denote the first zero of $\phi'_k$. The next step is to locate $\eta'_{1,n}$. For this claim, we start by using the change of function $U(\eta):= p^{1/2}(\eta) \phi_k(\eta)$ which transforms \eqref{gen_form} into the following equation on $U$,
$$ U'' + \Big[\frac{p'(\eta)^2-2p''(\eta)p(\eta)}{4p^2(\eta)} + \frac{q_k(\eta)}{p(\eta)}\Big]U = 0, \qquad \eta\in (0,1). $$
We should mention that $U$ and $\phi_k$ have the same zeros on $(0,1)$. \\
Straightforward computations show that $$\frac{p'(\eta)^2-2p''(\eta)p(\eta)}{p^2(\eta)}=\frac{1-\left(m+\frac{d}{2}\right)^2}{(1+\eta)^2}+\frac{4\left(m+\frac{d}{2}-1\right)}{(1+\eta)^2(1-\eta)}+\frac{4}{(1-\eta^2)^2} \geq (m+\frac{d}{2})(1-m-\frac{d}{2})+1 .$$\\
Since $ \displaystyle\frac{q_k(\eta)}{p(\eta)} = \frac{\alpha_k-\frac{c^2}{8}(1+\eta)}{1-\eta^2} \geq \alpha_k -\frac{c^2}{4} $, then  
$$\frac{p'(\eta)^2-2p''(\eta)p(\eta)}{4p^2(\eta)} + \frac{q_k(\eta)}{p(\eta)} \geq \alpha_k -C_{m,d}(c); \qquad C_{m,d}(c) = \frac{c^2}{4}+(m+\frac{d}{2})(m+\frac{d}{2}-1)-1.  $$
Thus, by Sturm comparison theorem, between two zeros of $\sin(\sqrt{\alpha_k -\frac{c^2}{4}}\eta)$ viewed as solution of $V''+\Big(\alpha_k -\frac{c^2}{4} \Big) V=0$, there exists a zero of $U$. Consequently, $\displaystyle \eta_{k,1} \leq \frac{\pi}{\sqrt{\alpha_k -\frac{c^2}{4}}}$ and  $ \displaystyle \eta'_{k,1} \leq \frac{2\pi}{\sqrt{\alpha_k -\frac{c^2}{4}}} = b_k$.\\
Let $\eta \in (0,\eta'_{n,1})$ and recall that 
	$$\Big((1-\eta^2)(1+\eta)^{m+\frac{d}{2}-1} \phi_k'(\eta)\Big)'=-\alpha_{k} \Big(1-\frac{c^2(\eta+1)}{8\alpha_{k}}\Big) (1+\eta)^{m+\frac{d}{2}-1}\phi_k(\eta).$$
	Then by integrating over $[\eta,\eta'_{n,1}]$, one gets 
	$$ \phi_k'(\eta) = -\frac{\alpha_{k}}{(1-\eta^2)(1+\eta)^{m+\frac{d}{2}-1}} \int_{\eta}^{\eta'_{n,1}} (1+t)^{m+\frac{d}{2}-1}\Big(1-\frac{c^2(t+1)}{8\alpha^{(c)}_{km}}\Big) \phi(t)dt. $$
	Using Hölder inequality and the normalization of $\phi_k$, one gets
	\begin{eqnarray}
		|\phi_k'(\eta) | &\leq& \frac{\alpha_{k}}{(1-\eta^2)(1+\eta)^{m+\frac{d}{2}-1}} \sqrt{2^{m+\frac{d}{2}-1}} \left|\int_{\eta}^{\eta'_{n,1}} (1+t)^{\frac{m+\frac{d}{2}-1}{2}}\phi(t)dt\right| \nonumber \\
		&\leq & \sqrt{2^{m+\frac{d}{2}-1}}\frac{\sqrt{\eta'_{k,1}}}{1-\eta'_{k,1}}\alpha_{k} \leq \sqrt{2^{m+\frac{d}{2}-1}}\frac{\sqrt{b_k}}{1-b_k}\alpha_{k}
	\end{eqnarray}
Then, 
$$ |\phi_k(\eta'_{n,1}) | \leq \sqrt{2^{m+\frac{d}{2}-1}} \frac{(b_k)^{3/2}}{1-b_{k}} \alpha_k \leq \frac{3\sqrt{3(m+\frac{d}{2}-1)}}{2} \left(\alpha_{k}\right)^{1/2} . $$
It may be useful to note that the last inequality follows from the fact that $ \displaystyle \alpha_{k}\geq \left(2/3\right)^6 \left(\frac{\pi}{m+\frac{d}{2}-1}\right)^2+\frac{c^2}{4}$.
One can summarize the above discussion as follows : 
$$ \max_{\g{x}\in \B^d} \left| \ps(\g{x}) \right| \leq \max_{\eta \in [-1,1]} \big|\phi^{(m,c)}_k(\eta) \big| \max_{\hat{\g{x}}\in \S^{d-1}}\left|Y^{(m)}_\ell(\hat{\g{x}})\right|  $$
The maximum of $\left|\phi^{(m,c)}_k\right| $ is attained in 1 or in some $\eta \in [a_{m,d},1]$. In both cases, $$ \max_{\eta \in [-1,1]}\big|\phi^{(m,c)}_k(\eta) \big| \leq \frac{3\sqrt{3}}{2}\sqrt{2^{m+\frac{d}{2}+1}(m+\frac{d}{2}-1)} \big(\alpha^{(c)}_{k,m}\big)^{1/2}$$
To conclude for the proof, it suffices to combine the previous inequality with \eqref{max_Y}.
\end{proof}
\subsection{Computation of ball spheroidal wave functions}
Note that $\ps \in L^2(\B^d)$ so that it can be expanded with respect to the ball polynomials basis as 
$$ \ps(\g{x}) = \sum_j \beta_j^{(k,m)} P^{(m)}_{j,l}(\g{x}).  $$
In \cite{Wang3}, authors have used the Bouwkamp method to compute ball PSWFs where it has been shown that the computation of these functions and their eigenvalues amounts to the determination of the eigenvectors and associate eigenvalues of a tridiagonal matrix. Then, it is interesting to study the spectral decay rate of the ball PSWFs expansion coefficients $(\beta_j^{(k,m)})_j$. For this purpose, we recall the finite Fourier transform of ball polynomials given by \eqref{fourier_fini}.
$$ 	\int_{\B^d} e^{-ic<\g{x},\g{y}>} P_{j,\ell}^{(m)}(\g{x}) d\g{x} = \frac{(2\pi)^{d/2} (-i)^m (-1)^j}{2(c\tau)^{\frac{d-1}{2}}} \frac{J_{2j+m+d/2(c \tau)}}{\sqrt{c\tau}}Y_{\ell}^{(m)}(\hat{\g{y}}).$$
Then, 
\begin{eqnarray}
\beta_j^{(k,m)} &=& \int_{\B^d} \ps(\g{x}) P_{j,\ell}^{(m)}(\g{x}) d\g{x}
= \frac{1}{\mu^{(m)}_k(c)} \int_{\B^d} \Bigg( \int_{\B^d} e^{-ic<\g{x},\g{y}>} P_{j,\ell}^{(m)} (\g{x}) d\g{x}\Bigg)\ps(\g{y}) d\g{y}  \nonumber \\
&=&  \frac{(2\pi)^{d/2}(-i)^m (-1)^j}{\mu^{(m)}_k(c) . c^{\frac{d-1}{2}} } \int_0^1 \tau^{m+\frac{d}{2}-\frac{1}{2}} \phi_k^{(m,c)}(2\tau^2-1) \frac{J_{2j+m+d/2(c\tau)}}{\sqrt{c\tau}} d\tau. \nonumber
\end{eqnarray}
Consequently, by Cauchy-Schwartz inequality together with the normalization of the radial part of ball spheroidal wave functions, one gets 
$$ |\beta_j^{(k,m)}| \leq \frac{(2\pi)^{d/2}\times c^{4j+2m+\frac{d-1}{2}}}{2^{4j+2m+d}(4j+2m+d)|\mu^{(m)}_k(c)|}\frac{1}{\Gamma^2(2j+m+\frac{d}{2}+1)}.$$
Recall that form \cite{Batir}, one has
\begin{equation}\label{batir}
\sqrt{2e} \Big(\frac{x+1/2}{e}\Big)^{1/2} \leq \Gamma(x+1) \leq \sqrt{2\pi} \Big(\frac{x+1/2}{e}\Big)^{1/2} .
\end{equation}
Then, by combining the two previous inequalities, one gets 
\begin{equation}
|\beta_j^{(k,m)}| \leq \frac{(2\pi)^{d/2}\times c^{2j+m}}{2^{4j+2m+d}(4j+2m+d)|\mu^{(m)}_k(c)|} \Bigg[\frac{ec}{2j+m+\frac{d+1}{2}}\Bigg]^{2j+m+\frac{d+1}{2}}.
\end{equation}
We may therefore summarize these calculations in the following proposition : 
\begin{proposition}
	For given real number $c>0$, let 
	$$\beta_j^{(k,m)} = \int_{\B^d} \ps(\g{x}) P_{j,\ell}^{(m)}(\g{x}) d\g{x}.$$
	Then we have 
	$$ |\beta_j^{(k,m)}|\leq \frac{(2\pi)^{d/2}\times c^{2j+m}}{2^{4j+2m+d}(4j+2m+d)|\mu^{(m)}_k(c)|} \Bigg[\frac{ec}{2j+m+\frac{d+1}{2}}\Bigg]^{2j+m+\frac{d+1}{2}}.$$
\end{proposition} 
\section{Approximation of almost band-limited functions over the d-dimensional unit ball }
The aim of this section is to study the quality of approximation in the framework of the d-dimensional ball prolate spheroidal wave functions and the ball polynomials series expansion. 
\subsection{Approximation by ball prolate spheroidal wave functions}
In this paragraph, we show that the ball prolate spheroidal wave functions are well adapted for the approximation of almost band-limited functions. For this claim, we start by proving that ball PSWFs are also well adapted to approach functions from the Paley-Wiener space $$\mathcal{B}_c = \{ f\in L^2(\R^d) : \mathcal{F}.f(u) = 0 \quad \forall u \not \in \B^d(0,c) \}.$$
Let us denote by $\displaystyle S^{(M)}_N.f := \sum_{m=0}^{M}\sum_{\ell=1}^{N(d,m)}\sum_{k=0}^{N}<f,\ps> \ps(\g{x}) $ the orthogonal projection of a function $f \in L^2(\B^d)$ on the span of the first ball prolate functions. We may now state our first approximation theorem :
\begin{theorem}
	Let $f\in L^2(\R^d)$ be an $\epsilon_c$-band-limited function with $\Omega = \B(0,c) $. Then, for any positive integer $N\geq ec/2$, we have 
	\begin{equation}
	\norm{f-S^{(M)}_Nf}_{L^2(\B^d)} \leq \left( 2\epsilon_c + C_M (\chi^{(M)}_N(c))^{1/2} |\lambda^{(M)}_N(c)| \right) \norm{f}_{L^2(\R^d)},
	\end{equation} 
where $\displaystyle C_M= \frac{\pi^{d/4}}{\sqrt{\Gamma(\frac{d}{2}+1)}} \frac{3\sqrt{3(M+\frac{d}{2}-1)}}{2}\sqrt{\frac{N(d,M)}{\Omega_{d-1}}}.$.
\end{theorem}
Note that the super-exponential decay rate of $\lambda^{(M)}_k$ described in theorem \ref{max} shows that the dependence on $M$ of $C_M$ and $\chi^{(M)}_N$ is insignificant. Thus,  ball prolates spheroidal wave functions are well adapted for the approximation of almost band-limited functions.
	\begin{proof} Let us first study the case where $f \in \mathcal{B}_c$.
	By Parseval's inequality, 
	\begin{equation}\label{1}
	\displaystyle \norm{f-S^{(M)}_N.f}^2_2 = \sum_{m>M}\sum_{k>N}\sum_{\ell=1}^{N(d,m)} |<f,\ps>|^2.
	\end{equation}
	Then, the main step in this proof is how to estimate $|<f,\ps>|$.\\
	Recall that from the Fourier inversion formula, one has $ \displaystyle f(\g{x}) = \Big(\frac{c}{2\pi}\Big)^d \int_{\B^d} e^{ic<\g{x},\g{y}>} \hat{f}(\g{y}) d\g{y} $. Consequently, for any positive integer $k$, we have
	\begin{eqnarray}\label{2}
	<f,\ps> &=& \int_{\B^d}f(\g{x}) \ps(\g{x}) d\g{x} = \Big(\frac{c}{2\pi}\Big)^d \int_{\B^d} \ps(\g{x}) \int_{\B^d} e^{ic<\g{x},\g{y}>} \hat{f}(\g{y}) d\g{y} d\g{x} \nonumber \\
	&=& \Big(\frac{c}{2\pi}\Big)^d\int_{\B^d}\hat{f}(\g{y}) \int_{\B^d} e^{ic<\g{x},\g{y}>}\ps(\g{x})d\g{x} d\g{y} \nonumber \\
	&=& \Big(\frac{c}{2\pi}\Big)^d |\mu^{(m)}_k(c)| \int_{\B^d}\hat{f}(\g{y})\ps(\g{y}) d\g{y} \nonumber \\
	&\leq& \Big(\frac{c}{2\pi}\Big)^d \frac{\pi^{d/4}}{\sqrt{\Gamma(d/2+1)}} |\mu^{(m)}_k(c)| \sup_{y\in \B^d} |\ps(\g{y})| \norm{f}_{L^2(\R^d)} \nonumber \\
	&\leq& C_m\Big(\frac{c}{2\pi}\Big)^d \frac{\pi^{d/4}}{\sqrt{\Gamma(d/2+1)}} |\mu^{(m)}_k(c)|(\chi^{(m)}_k)^{1/2} \norm{f}_{L^2(\R^d)}.
	\end{eqnarray}
Note that the last inequality follows from theorem \ref{max}. Finally, by combining the last inequality together with \eqref{1}, \eqref{2} as well as the super-exponential decay rate of $\lambda^{(m)}_k$ , one gets 
	$$  \norm{f-S_Nf}_{L^2(\B^d)} \leq \left( C |\mu_N(c)| (\chi_N(c))^{1/2} \right) \norm{f}_{L^2(\R^d)}. $$
    Now, let $f$ be an $\epsilon_c$-band-limited function. \\
	We have
	$$ \norm{f-S_Nf}_{L^2(\B^d)} \leq \norm{f-B_cf}_{L^2(\B^d)} + \norm{B_cf-S_NB_cf}_{L^2(\B^d)} + \norm{S_N[B_cf-f]}_{L^2(\B^d)}. $$
	On the other hand, since $f$ is $\epsilon_c$-band-limited then $\norm{B_cf-f}_{L^2(\B^d)} \leq \epsilon_c\norm{f}_{L^2(\R^d)}$.
	Using the fact that $S_N$ is a contraction, one gets 
	$$ \norm{S_N[B_cf-f]}_{L^2(\B^d)} \leq \norm{B_cf-f}_{L^2(\B^d)} \leq \epsilon_c \norm{f}_{L^2(\R^d)} . $$
	Finally, we apply the previous proposition to the second term by noticing that $ B_cf $ is a band-limited function.
\end{proof}

\subsection{Approximation of almost band limited functions by ball polynomials}
We study in this section the quality of approximation of almost band limited functions by their expansion in the basis of ball polynomials. To do so, we start by the following technical lemma controlling the general term of the associated projection series. 
\begin{lemma}
	Let $c>0$. Then, for any $f\in B_c$ and any $ k \geq \frac{ec}{2} $
	\begin{equation}\label{ineq_ball}
		\left| <f, \jp>_{L^2(\B^d)} \right| \leq \frac{1}{2^{2k+m+\frac{d}{2}+1}\sqrt{2ec(4k+3m+d)}}
		\left( \frac{ec}{2k+m+\frac{d+1}{2}} \right)^{2k+m+\frac{d+1}{2}}.
	\end{equation}
\end{lemma}

\begin{proof}
	Let $f\in B_c$ be a band-limited function. Then, by the Fourier inversion formula, we have 
	$$ f(x) = \frac{1}{(2\pi)^{d/2}} \int_{\B^d(0,c)} \hat{f}(\g{y}) e^{i<\g{x},\g{y}>} d\g{y} = \Big(\frac{c^2}{2\pi}\Big)^{d/2} \int_{\B^d} \hat{f}(c\g{y}) e^{ic<\g{x},\g{y}>} d\g{y} .$$
	Recall that the finite Fourier transform of the ball polynomials is given by 
	$$ \int_{\B^d} e^{ic<\g{x},\g{y}>} \jp(\g{y}) d\g{y} = (2\pi)^{d/2} \frac{(-i)^m (-1)^j}{2} \frac{J_{2k+m+d/2}(c\rho)}{(c\rho)^{d/2}} Y^m_{\ell} (\hat{\g{x}}) \qquad \g{x}=\rho \hat{\g{x}}. $$
	Then, by these last two relations,
	\begin{eqnarray}
		<f,\jp> &=& \int_{\B^d} f(\g{x}) \jp(\g{x}) d\g{x} = \left( \frac{c^2}{2\pi} \right)^{d/2} \int_{\B^d} \int_{\B^d} \hat{f}(cy) e^{ic<\g{x},\g{y}>} d\g{y} \jp(\g{x}) d\g{x} \nonumber \\
		&=& c^d \frac{(-i)^m (-1)^j}{2} \int_{0}^{1} \rho^{m+d-1} \int_{\S^{d-1}} \hat{f}(c\rho \hat{\g{y}} ) \frac{J_{2k+m+d/2}(c\rho)}{(c\rho)^{d/2}}Y^m_{\ell} (\hat{\g{x}}) d\sigma(\hat{\g{y}}) d\rho.
	\end{eqnarray}
By \eqref{estim1}, one gets 
$$ \left| \frac{J_{2k+m+d/2}(c\rho)}{(c\rho)^{d/2}} \right| \leq \frac{(c\rho)^{2k+m}}{2^{2k+m+\frac{d}{2}} \Gamma(2k+m+\frac{d}{2}+1) },$$
Then,
\begin{equation}
	\left| <f,\jp> \right| \leq \frac{c^{2k+2m+d-1}}{2^{2k+m+1+d/2} \Gamma(2k+m+1+d/2)} \int_0^1 \rho^{2k+2m+d-1} \left|\int_{\S^{d-1}} \hat{f}(c\rho \hat{\g{y}}) Y^m_\ell(\hat{\g{y}})  d\sigma(\hat{\g{y}})\right| d\rho.
\end{equation}
On the other hand 
$$ \left| \int_{\S^{d-1}}\hat{f}(c\rho \hat{\g{y}}) Y^m_\ell(\hat{\g{y}})  d\sigma(\hat{\g{y}})  \right| \leq \left( \int_{\S^{d-1}} \left| \hat{f}(c\rho \hat{\g{y}}) \right|^2 d\sigma(\hat{\g{y}}) \right)^{1/2} $$
Now let $ \displaystyle g(\rho) := \rho^{\frac{m+d-1}{2}} \left( \int_{\S^{d-1}} \left| \hat{f}(c\rho \hat{\g{y}}) \right|^2 d\sigma(\hat{\g{y}}) \right)^{1/2}$. We remark that 
\begin{equation}
	\int_{0}^{1} \left| g(\rho) \right|^2 d\rho = \int_{\B^d} \left| \hat{f}(c\g{y}) \right|^2 d\g{y} = \frac{1}{c^{d}} \norm{f}^2_2.
\end{equation}
Therefore, 
$$	\left| <f,\jp> \right| \leq  \frac{c^{2k+m+d/2}}{2^{2k+m+1+\frac{d}{2}} \sqrt{4k+3m+d}\Gamma(2k+m+1+\frac{d}{2})} \norm{f}_2. $$
To conclude for the proof of this lemma, it remains to estimate $\Gamma(2k+m+1+\frac{d}{2})$ via Batir's inequality \eqref{batir}.
\end{proof}
 Analogously to the previous case, we define $\displaystyle \Pi^{(M)}_N.f := \sum_{m=0}^{M}\sum_{\ell=1}^{N(d,m)}\sum_{k=0}^{N} <f,\jp> \jp$ the orthogonal projection on the span of the first ball polynomials.
We have now all the ingredients to state our second approximation theorem.
\begin{theorem}
	Let $f\in L^2(\R^d)$ be an $\epsilon_c$-band-limited function. Then, for any positive integer $N\geq \frac{ec-\frac{d+1}{2}}{2}$, we have 
	\begin{equation}\label{approx2}
		\norm{f-\Pi^{(M)}_Nf}_{L^2(\B^d)} \leq \left( 2 \epsilon_c + C_N \left( \frac{ec}{2(N+1)+(M+1)+\frac{d+1}{2}} \right)^{2(N+1)+(M+1)+\frac{d+1}{2}}\right) \norm{f}_{L^2(\R^d)}.
	\end{equation} 
Here $C_N := \frac{1}{2^{2N+M+\frac{d}{2}+3}\sqrt{ec(4N+3M+d+4)}}\left[1+\frac{1}{4\ln \left( \frac{ec}{2N+M+2+\frac{d+1}{2}} \right)}\right]^{1/2}.$
\end{theorem}
\begin{proof}
	Let $f\in \mathcal{B}_c$. Note that for $\g{x}\in \B^d$, $$ \displaystyle f(x) - \Pi_N.f(x) = \sum_{m>M}\sum_{k>N}\sum_{\ell=1}^{N(d,m)} <f,\jp> \jp(x). $$ Then,
	\begin{eqnarray}
		\norm{f-\Pi_N.f}_{L^2(\B^d)}^2 &=& \sum_{m>M}\sum_{k>N}\sum_{\ell=1}^{N(d,m)} |<f,\jp>|^2 \nonumber\\
		&\leq& \sum_{m>M}\sum_{k>N}\sum_{\ell=1}^{N(d,m)} \frac{1}{2^{4k+2m+d+2}ec(4k+3m+d)}
		\left( \frac{ec}{2k+m+\frac{d+1}{2}} \right)^{4k+2m+d+1} \nonumber \\
		&\leq & \frac{1}{2^{4N+2M+d+6}ec(4N+3M+d+4)}\sum_{m>M}\sum_{k>N}\sum_{\ell=1}^{N(d,m)} \left( \frac{ec}{2k+m+\frac{d+1}{2}} \right)^{4k+2m+d+1}.
	\end{eqnarray}
Remark that $k$ and $m$ behave in the same way in the main decay factor. To simplify the proof and make it clearer, we will see only what happens with respect to $k$ and the proof of the behavior with respect to $m$ is almost identical. Let us treat the main factor separately, 
\begin{eqnarray}
	& &\sum_{k=N+1}^\infty \left( \frac{ec}{2k+m+\frac{d+1}{2}} \right)^{4k+2m+d+1} \nonumber \\ &=& 
	\left( \frac{ec}{2N+m+2+\frac{d+1}{2}} \right)^{4N+2m+d+5}+\sum_{k=N+2}^\infty \left( \frac{ec}{2k+m+\frac{d+1}{2}} \right)^{4k+2m+d+1} \nonumber \\
	&\leq& \left( \frac{ec}{2N+m+2+\frac{d+1}{2}} \right)^{4N+2m+d+5} + \int_{N+1}^\infty \left( \frac{ec}{2(N+2)+m+\frac{d+1}{2}} \right)^{4x+2m+d+1} dx \nonumber \\
	&=& \left( \frac{ec}{2N+m+2+\frac{d+1}{2}} \right)^{4N+2m+d+5} + \nonumber\\ & & \exp\left[(2m+d+1) \ln \left( \frac{ec}{2N+m+2+\frac{d+1}{2}} \right) \right] \int_{N+1}^{\infty} \exp\left[4x \ln \left( \frac{ec}{2N+m+2+\frac{d+1}{2}} \right) \right] dx \nonumber \\
	&\leq& \left[1+\frac{1}{4\ln \left( \frac{ec}{2N+m+2+\frac{d+1}{2}} \right)}\right] \left( \frac{ec}{2N+m+2+\frac{d+1}{2}} \right)^{4N+2m+d+5}.  \nonumber
\end{eqnarray}
Thus, for all $f\in \mathcal{B}_c$
\begin{equation}\label{last_ineq}
	\norm{f-\Pi^{(M)}_Nf}_{L^2(\B^d)} \leq  C_{N,M} \left( \frac{ec}{2(N+1)+(M+1)+\frac{d+1}{2}} \right)^{2(N+1)+(M+1)+\frac{d+1}{2}} \norm{f}_{L^2(\R^d)}.
\end{equation}
 Now, let $f$ be an $\epsilon_c$-band-limited function. \\
$$ \norm{f-\Pi^{(M)}_Nf}_{L^2(\B^d)} \leq \norm{f-B_cf}_{L^2(\B^d)} + \norm{B_cf-\Pi^{(M)}_NB_cf}_{L^2(\B^d)} + \norm{\Pi^{(M)}_N[B_cf-f]}_{L^2(\B^d)}. $$
On the other hand, since $f$ is $\epsilon_c$-band-limited, then $\norm{B_cf-f}_{L^2(\B^d)} \leq \epsilon_c\norm{f}_{L^2(\R^d)}$.
Using the fact that $\Pi^{(M)}_N$ is a contraction, one gets 
$$ \norm{\Pi^{(M)}_N[B_cf-f]}_{L^2(\B^d)} \leq \norm{B_cf-f}_{L^2(\B^d)} \leq \epsilon_c \norm{f}_{L^2(\R^d)} . $$
By the previous analysis and \eqref{last_ineq}, \eqref{approx2} follows at once.
\end{proof}

\end{document}